\documentclass[12pt]{amsart}

\usepackage[margin=1.15in]{geometry}

\usepackage{amsmath,amscd,amssymb,amsfonts,latexsym,wasysym, mathrsfs, mathtools,hhline,color}
\usepackage[all, cmtip]{xy}
\usepackage{csquotes}
\usepackage{url}

\definecolor{hot}{RGB}{65,105,225}

\usepackage[pagebackref=true,colorlinks=true, linkcolor=hot ,  citecolor=hot, urlcolor=hot]{hyperref}

\usepackage{ textcomp }
\usepackage{ tipa }

\usepackage{graphicx,enumerate}

\theoremstyle{plain}
\newtheorem{theorem}{Theorem}[]
\newtheorem{prop}[theorem]{Proposition}

\newtheorem{lm}[theorem]{Lemma}

\newtheorem{cor}[theorem]{Corollary}

\newtheorem{lemma}[theorem]{Lemma}
\newtheorem{thrm}[theorem]{Theorem}

\theoremstyle{definition}

\newtheorem{defn}[theorem]{Definition}

\newtheorem{rmk}[theorem]{Remark}

\newtheorem{ex}[theorem]{Example}
\newtheorem*{ex*}{Example}

\newtheorem{question}[theorem]{Question}

\def\be{\begin{equation}}
\def\ee{\end{equation}}

\def\bt{\begin{thrm}}
\def\et{\end{thrm}}

\def\bc{\begin{cor}}
\def\ec{\end{cor}}

\def\br{\begin{rmk}}
\def\er{\end{rmk}}

\def\bp{\begin{prop}}
\def\ep{\end{prop}}

\def\bl{\begin{lm}}
\def\el{\end{lm}}

\def\bex{\begin{ex}}
\def\eex{\end{ex}}

\def\bd{\begin{defn}}
\def\ed{\end{defn}}

\newcommand\sO{{\mathcal O}}

\newcommand\sH{{\mathcal H}}

\newcommand\sF{{\mathcal F}}

\DeclareMathOperator{\id}{id}                    

\DeclareMathOperator{\coker}{coker}
\DeclareMathOperator{\Ext}{Ext}

\def\bC{\mathbb{C}}

\def\bP{\mathbb{P}}
\def\bA{\mathbb{A}}

\def\bQ{\mathbb{Q}}

\def\bZ{\mathbb{Z}}

\def\bD{\mathbb{D}}

\title[Lyubeznik numbers of irreducible projective varieties]{Lyubeznik numbers of irreducible projective varieties depend on the embedding}

\author{Botong Wang}
\address{Department of Mathematics,         University of Wisconsin-Madison,  480 Lincoln Drive, Madison WI 53706-1388, USA.}
\email {wang@math.wisc.edu}
\date{\today}



\begin{document}

\maketitle
\begin{abstract}
We construct irreducible complex projective varieties such that the Lyubeznik numbers of their affine cones depend on the choices of projective embeddings. The main ingredient is the recent work of Reichelt-Saito-Walther, where the Lyubeznik numbers are reinterpreted using perverse sheaves and reducible projective varieties with dependent Lyubeznik numbers are constructed. 
\end{abstract}


\section{Introduction}\label{intro}
In the seminal paper \cite{L}, Lyubeznik studied finiteness properties of local cohomology modules. While the article was primarily concerned with the characteristic zero case, when combined with earlier results of Huneke and Sharp \cite{HS} it allowed Lyubeznik to
define a set of numerical invariants attached to a point on the spectrum of a finitely generated algebra over any field by way of local cohomology. These numbers became
known as Lyubeznik numbers and can be determined by viewing the given
ring as a quotient of a finite polynomial ring over the base field. Lyubeznik's
remarkable theorem is that the numbers are invariants of the original ring, and does not depend on the presentation.


While defined in purely algebraic ways, it was realized early on that in characteristic zero these numbers often had something to do with topology. For the vertex of a cones over smooth projective varieties, Garcia-Lopez and Sabbah \cite{GLS} gave a formula in terms of differences of Betti numbers of the projective variety. This idea was sharpened later
by Switala \cite{S}. Independently, Walther \cite{W} and Kawasaki \cite{K} studied the relation of Lyubeznik numbers with the  topology of the punctured spectrum for local rings of Krull dimension two. Recently, Lorincz and Raicu \cite{LR} obtained combinatorial formulas of the Lyubeznik numbers for determinantal rings. See \cite{NWZ} for a recent survey of the subject. 

In this paper, we study the Lyubeznik numbers of the local rings of an affine cone at the cone point. Let $X$ be a complex projective variety, and let $C(X)\subset \bA_{\bC}^{n+1}$ be the affine cone of $X$ with respect to some projective embedding $X\to \bP_\bC^n$. Denote the coordinate ring of $\bA_\bC^{n+1}$ by $R$, and denote the defining ideal of $C(X)$ by $I$. The Lyubeznik numbers $\lambda_{k, j}(C(X))$ of the affine cone $C(X)$ are defined by
\[
\lambda_{k,j}(C(X))\stackrel{\text{def}}{=}\dim_\bC \Ext^k_R (\bC, H^{n-j}_I R)
\]
for $k, j\in \bZ_{\geq 0}$. An early question, predating the year 2000, was whether the Lyubeznik numbers $\lambda_{k,j}(C(X))$ depend on the choice of the projective embedding of $X$. 

In positive characteristic, it is proved by Wenliang Zhang \cite{Z} that the Lyubeznik numbers of $C(X)$ depend only on $X$, not the choice of embedding. In \cite{GLS} and \cite{S}, the same result is proved for smooth projective varieties in characteristic zero.  More recently, Thomas Reichelt, Morihiko Saito and Uli Walther \cite{RSW} constructed reducible complex projective varieties $X$ such that the Lyubeznik numbers of $C(X)$ depend on the choices of projective embeddings. In this paper, we extend their result to irreducible projective varieties.

\begin{defn}
	Let $X$ be a complex projective variety. We say that $X$ has \textbf{dependent $(k, j)$-Lyubeznik number}, if the Lyubeznik number $\lambda_{k, j}(C(X))$ of the affine cone of $X$ depend on the choice of the projective embedding. We say that $X$ has \textbf{dependent Lyubeznik numbers}, if $X$ has dependent $(k, j)$-Lyubeznik number for some $k, j$. 
\end{defn}
The main result of this paper is the following. 
\begin{theorem}\label{thm_main}
There exist irreducible projective varieties with dependent Lyubeznik numbers. 
\end{theorem}
The proof of the theorem strongly relies the results of \cite{RSW}, where a necessary condition of dependent Lyubeznik numbers is given (Corollary \ref{cor_RSW}) and examples of reducible projective varieties satisfying such necessary condition are constructed (Proposition \ref{prop_ex}). Given such a reducible projective variety $X$, we construct an irreducible projective variety $Y$ that carries over the necessary condition of dependent Lyubeznik numbers from $X$ (Proposition \ref{prop_constr}). 

\textbf{Acknowledgment} We thank the anonymous referee for his/her detailed comments, especially about the history of the Lyubeznik numbers. We also thank Claudiu Raicu, Laurentiu Maxim and Uli Walther for helpful discussions. The author is partially supported by the NSF grant DMS-1701305.

\section{The results of Reichelt-Saito-Walther}
In this section, we briefly recall the main results of Reichelt-Saito-Walther (\cite{RSW}). 
The first one  is that the Lyubeznik numbers can be computed using certain perverse sheaves. 
\begin{prop}\label{prop_RSW}\cite[Proposition 1]{RSW}
	Let $C(X)$ be the affine cone of a projective variety $X$ with respect to some projective embedding. Then
	\[
	\lambda_{k, j}(C(X))=\dim_\bQ H^k i^!\big(\,^p\sH^{-j}(\bD\bQ_C)\big)
	\]
	where the map $i: 0\to C(X)$ is the inclusion of the origin to the affine cone, the functor $\bD: D^b_c(\bQ_{C(X)})\to D^b_c(\bQ_{C(X)})$ is the duality functor in the derived category of $\bQ$-constructible sheaves on $C(X)$ and $^p\sH^{-j}$ is taking the $(-j)$-th perverse cohomology\footnote{See \cite{D} for an introduction about constructible sheaves, the duality functor and perverse sheaves.}. 
\end{prop}

\begin{defn}
	Let $X$ be a projective variety of dimension $d$ and let $\sF$ be a perverse sheaf on $X$. We say $\sF$ is \textbf{Lefschetz dependent in degree $k$} if the quantity
	\begin{equation*}
	\dim \ker\big(A: H^{k-1}(X, \sF)\to H^{k+1}(X, \sF)\big)+\dim \coker\big(A: H^{k-2}(X, \sF)\to H^k(X, \sF)\big)
	\end{equation*}
	depends on the choice of the ample class $A$ of $X$. 
\end{defn}
\begin{rmk}
	By the semi-continuity theorem, $\sF$ is Lefschetz dependent in degree $k$ if and only if one of 
	$$\dim \ker\big(A: H^{k-1}(X, \sF)\to H^{k+1}(X, \sF)\big)$$ and $$\dim \coker\big(A: H^{k-2}(X, \sF)\to H^k(X, \sF)\big)$$ depends on the choice of the ample class $A$.
\end{rmk}

As a consequence of Proposition \ref{prop_RSW}, we have the following. 
\begin{cor}\label{cor_RSW}\cite[Corollary 1]{RSW}
For a projective variety $X$, the Lyubeznik number $\lambda_{k, j}(C(X))$ of the affine cone of $X$ depends on the choice of the projective embedding if the perverse sheaf $^p\sH^{1-j}(\bD\bQ_X)$ is Lefschetz dependent in degree $k$ for some $k\geq 2, j\geq 1$. 
\end{cor}
In \cite{RSW}, reducible projective varieties with dependent Lyubeznik numbers were constructed. We will not repeat their construction. Instead, we state it as the following proposition. 
\begin{prop}\label{prop_ex}\cite[Section 2.1, 2.2]{RSW}
	For any integers $d_2>d_1\geq 2$, there exists a non-equidimensional projective variety $X$ with irreducible components of dimension $d_1+d_2$ and $d_1+d_2+1$, such that the perverse sheaf $^p\sH^{-d_1-d_2}(\bD\bQ_{C(X)})$ is Lefschetz dependent in degree $d_2-d_1+1$. 
\end{prop}
By Corollary \ref{cor_RSW}, the varieties in Proposition \ref{prop_ex} have dependent Lyubeznik numbers. In \cite[Section 2.3]{RSW}, much more complicated equidimensional reducible projective varieties with dependent Lyubeznik numbers were constructed. For the purpose of this paper, the non-equidimensional examples are sufficient. 



\section{The irreducible examples}
We state the desired construction as the following proposition, which may also be interesting for other purposes. 
\begin{prop}\label{prop_constr}
	Given any (possibly reducible and non-equidimensional) complex projective variety $X$ with two ample classes $A_1$ and $A_2$, there exists an irreducible projective variety $Y$ and a closed embedding $f: X\to Y$ such that
	\begin{enumerate}
		\item\label{i1} $f$ induces an isomorphism between $X$ and the singular locus of $Y$ (with the induced reduced structure);
		\item\label{i2} $^p\sH^i(\bD\bQ_Y)$ is isomorphic to $^p\sH^{i+1}\big(\bD\bQ_{f(X)}\big)$ for $i>-\dim Y$;
		\item\label{i3} there exist ample classes $B_1, B_2$ of $Y$ such that $f^*B_1$ is a multiple of $A_1$ and $f^*B_2$ is a multiple of $A_2$. 
	\end{enumerate}
\end{prop}
We first construct the variety $Y$ and the map $f$. By Serre's vanishing theorem, for sufficiently large integer $m$, both $mA_1-A_2$ and $mA_2-A_1$ are very ample classes. Let $A_1'=mA_1-A_2$ and $A_2'=mA_2-A_1$. Let $\varphi_1: X\to \bP^{N_1}$ and $\varphi_2: X\to \bP^{N_2}$ be the projective embedding maps associated to very ample classes $A'_1$ and $A'_2$ respectively. Define $\varphi\stackrel{\text{def}}{=}(\varphi_1, \varphi_2): X\to \bP^{N_1}\times \bP^{N_2}$. 
\begin{lemma}
	Under the above notations, there exists a vector bundle $E$ on $\bP^{N_1}\times \bP^{N_2}$ and a global section $s$ of $E$ such that the zero locus of $s$ is equal to $\varphi(X)$. 
\end{lemma}
\begin{proof}
	Suppose $\varphi(X)$ is equal to the intersection of divisors $D_1, \ldots, D_l$ of $\bP^{N_1}\times \bP^{N_2}$. For each $D_i$, there is a global section $s_i$ of $\sO(D_i)$ whose zero locus is equal to $D_i$. The section $s=(s_1, \ldots, s_l)$ of the vector bundle $E=\sO(D_1)\oplus \cdots \oplus \sO(D_l)$ has zero locus $\varphi(X)$. 
\end{proof}

Let $E$ and $s$ be as in the preceding lemma. Consider $E$ as a quasi-projective variety. We define a subvariety $Y'$ of $E\times \bC$ by $Y'=\{(\eta, \lambda)\in E\times \bC\mid \eta\in \Gamma(\lambda s)\}$, where $\Gamma(\lambda s)$ is the image of the section $\lambda s$ in $E$. Let $C$ be the plane conic curve defined by $y^2=x^2+x^3$, and let $\pi: \bC\to C, \lambda\mapsto (\lambda^2-1, \lambda^3-\lambda)$ be the normalization map. Denote the nodal point of $C$ by $Q$. Then $\pi$ maps both $-1$ and $1$ to the nodal point, and $\pi$ is an isomorphism away from $Q$. 

We define $Y''$ to be the image of $Y'$ under the map $(\id, \pi): E\times \bC\to E\times C$. There exists a natural map $g': Y''\to \bP^{N_1}\times \bP^{N_2}$ as the composition of the projection to $E$ and the vector bundle map $E\to \bP^{N_1}\times \bP^{N_2}$. Notice that the singular locus of $Y''$ is contained in $E\times \{Q\}$, and hence closed in $Y''$. Thus, we can construct a projective compactification $Y$ of $Y''$ such that $Y$ is smooth along $Y\setminus Y''$. Additionally, using the resolution of singularty, we can assume that the map $g'$ extends to $g: Y\to \bP^{N_1}\times \bP^{N_2}$. 

Let $\psi: X\to E$ be the composition of $\varphi: X\to \bP^{N_1}\times \bP^{N_2}$ and the embedding $\bP^{N_1}\times \bP^{N_2}\to E$ as the zero section. We define the map $f: X\to Y$ by $f(P)=(\psi(P), Q)$. Next, we prove $Y$ and $f$ satisfy the there properties in Proposition \ref{prop_constr}. 

\begin{proof}[Proof of Property (\ref{i1})]
	By our construction, the restriction $(\id_E, \pi)|_{Y'}: Y'\to Y''$ is an isomorphism away from $\psi(X)\times \{Q\}$, and it is two-to-one along $\psi(X)\times \{Q\}$. Therefore, $Y''$ is smooth away from $\psi(X)\times \{Q\}$. The variety $Y''$ is singular along $\psi(X)\times \{Q\}$, because it is analytically locally reducible there. Since $Y$ is smooth along $Y\setminus Y''$, property (\ref{i1}) follows. 
\end{proof}
\begin{proof}[Proof of Property (\ref{i2})]
	Denote the normalization of $Y$ by $\tilde{Y}$, and denote the normalization map by $q: \tilde{Y}\to Y$. By the construction, $q$ is isomorphic to the map $(\id_E, \pi)|_{Y'}: Y'\to Y''$ over $Y''$, and it induces an isomorphism over $Y\setminus Y''$. Therefore, we have a short exact sequence of constructible sheaves,
	\begin{equation*}
	0\to \bQ_Y \to q_*\bQ_{\tilde{Y}}\to \bQ_{f(X)}\to 0.
	\end{equation*}
	Applying the Verdier duality functor $\bD$, we have a distinguished triangle
	$$
	\bD \bQ_{f(X)}\to \bD (q_*\bQ_{\tilde{Y}})\to \bD \bQ_Y\to\bD \bQ_{f(X)}[1].
	$$
	Taking perverse cohomology gives rise to a long exact sequence
	$$\cdots\to  \,^p\sH^i(\bD (q_*\bQ_{\tilde{Y}}))\to \,^p\sH^i(\bD \bQ_Y)\to \,^p\sH^{i+1}(\bD \bQ_{f(X)})\to \,^p\sH^{i+1}(\bD q_*\bQ_{\tilde{Y}})\to \cdots$$
	By definition, the normalization of $Y''$ is isomorphic to $Y'$, which is smooth. Moreover, $Y$ is smooth along $Y\setminus Y''$. Therefore, $\tilde{Y}$ is smooth, and hence $\bQ_{\tilde{Y}}[\dim Y]$ is perverse. Since $q$ is a finite morphism, $q_*\bQ_{\tilde{Y}}[\dim Y]$ is also perverse. Since the duality functor preserves perverse sheaves, $^p\sH^i(\bD (q_*\bQ_{\tilde{Y}}))=0$ for any $i\neq -\dim Y$. By the long exact sequence, for any $i>-\dim Y$, we have $^p\sH^i(\bD \bQ_Y)\cong \,^p\sH^{i+1}(\bD \bQ_{f(X)})$. 
\end{proof}
\begin{proof}[Proof of Property (\ref{i3})]
	Choose an ample class $B$ of $Y$. Since $\pi(1)=Q$, the inclusion $\varphi(X)\subset Y$ factors as
	$$\varphi(X)\subset \Gamma(s)\times \{Q\}\subset Y''\subset Y.$$
	Therefore, the ample class $f^*B$ is equal to the pullback of an ample class on $\Gamma(s)\times \{Q\}$.
	Notice that $\Gamma(s)$, as the total space of a section of $E$, is naturally isomorphic to $\bP^{N_1}\times \bP^{N_2}$. Therefore, $f^*B$ is a positive combination of $A_1'$ and $A_2'$, i.e., $f^*B=\alpha A_1+\beta A_2$ for some $\alpha, \beta\in \bZ_{>0}$. 
	
	Let $g_1$ (resp. $g_2$) be the composition of $g: Y\to \bP^{N_1}\times \bP^{N_2}$ and the projection $\bP^{N_1}\times \bP^{N_2}\to \bP^{N_1}$ (resp. $\bP^{N_1}\times \bP^{N_2}\to \bP^{N_2}$). Denote the hyperplane classes of $\bP^{N_1}$ and $\bP^{N_2}$ by $H_1$ and $H_2$ respectively. Since $g_1^*H_1$ and $g_2^*H_2$ are the divisor classes of base-point-free line bundles on $Y$, the divisor class $B+\mu g_1^*H_1+\nu g_2^*H_2$ is ample on $Y$ for any $\mu, \nu\in \bZ_{\geq 0}$. 
	
	Since $f^*(g_1^*H_1)=A_1'$ and $f^*(g_2^*H_2)=A_2'$, we have
	\[
	f^*(B+\alpha g_2^*H_2)=\alpha A_1+\beta A_2+\alpha(m A_2-A_1)=(\beta+\alpha m)A_2
	\]
	is a multiple of $A_2$. Thus, $B_2\stackrel{\text{def}}{=}B+\alpha g_2^*H_2$ satisfies the desired properties. Similarly, we can let $B_1\stackrel{\text{def}}{=}B+\beta g_1^*H_1$. 
\end{proof}

\begin{cor}\label{cor_dep}
	Suppose $X$ is a (possibly reducible and non-equidimensional) projective variety such that $^p\sH^{1-j}(\bD\bQ_X)$ is Lefschetz dependent in degree $k$ with $k\geq 2, j\geq 1$.
	Let $Y$ be as in the preceding proposition with the appropriate choice of $A_1$ and $A_2$. Then $Y$ has dependent $(k, j+1)$-Lyubeznik number. 
\end{cor}
\begin{proof}
	Since $f$ is a closed embedding, it induces an isomorphism between $X$ and $f(X)$. By properties (\ref{i2}) and (\ref{i3}), with the appropriate choice of $A_1$ and $A_2$, we know $^p\sH^{-j}(\bD\bQ_Y)$ is Lefschetz dependent in degree $k$. By Corollary \ref{cor_RSW}, the variety $Y$ has dependent $(k, j+1)$-Lyubeznik number.
\end{proof}

\begin{proof}[Proof of Theorem \ref{thm_main}]
	The theorem follows immediately from Proposition \ref{prop_ex} and Corollary \ref{cor_dep}. 
\end{proof}

All the examples of varieties with dependent Lyubeznik numbers here and in \cite{RSW} are obtained by gluing smooth varieties along subvarieties. In particular, they are not analytically locally irreducible. So a challenging question is the following. 

\begin{question}
	Does there exist a normal projective variety $X$, such that the Lyubeznik numbers of the affine cone $\lambda_{k, j}(C(X))$ depend on the choice of the projective embedding of $X$?
\end{question}



\begin{thebibliography}{ADMSPR}
	

\bibitem[Dim]{D} A. Dimca, {\it Sheaves in topology},  Universitext. Springer-Verlag, Berlin, 2004.

\bibitem[GLS]{GLS} R. Garc\'ia L\'opez, C. Sabbah, {\it Topological computation of local cohomology multiplicities.} Dedicated to the memory of Fernando Serrano. Collect. Math. 49 (1998), no. 2-3, 317-324. 

\bibitem[HS]{HS} C. Huneke, R. Sharp, {\it Bass numbers of local cohomology modules.} Trans. Amer. Math. Soc. 339 (1993), no. 2, 765-779. 

\bibitem[Kaw]{K} K. Kawasaki, {\it On the Lyubeznik number of local cohomology modules.} Bull. Nara Univ. Ed. Natur. Sci. 49 (2000), no. 2, 5-7.

\bibitem[LR]{LR} A. L\"orincz, C. Raicu, {\it Iterated local cohomology groups and Lyubeznik numbers for determinantal rings.} arXiv:1805.08895

\bibitem[Lyu]{L} G. Lyubeznik, {\it Finiteness properties of local cohomology modules (an application of $D$-modules to commutative algebra)}, Invent. Math. 113 (1993), no. 1, 41-55. 

\bibitem[NWZ]{NWZ} L. N\'u\~nez-Betancourt, E.E. Witt, W. Zhang, {\it A survey on the Lyubeznik numbers.}  Mexican mathematicians abroad: recent contributions, 137-163, Contemp. Math., 657, Aportaciones Mat., Amer. Math. Soc., Providence, RI, 2016. 

\bibitem[RSW]{RSW} T. Reichelt, M. Saito, U. Walther {\it Dependence of Lyubeznik numbers of cones of projective embeddings}, arXiv:1803.07448

\bibitem[Swi]{S} N. Switala, {\it Lyubeznik numbers for nonsingular projective varieties}, Bull. Lond. Math. Soc. 47 (2015), no. 1, 1-6. 

\bibitem[Wal]{W} U. Walther, {\it On the Lyubeznik numbers of a local ring.} Proc. Amer. Math. Soc. 129 (2001), no. 6, 1631-1634.

\bibitem[Zha]{Z} W. Zhang, {\it Lyubeznik numbers of projective schemes}, Adv. Math. 228 (2011), no. 1, 575-616. 

\end{thebibliography}
\end{document}